\documentclass[11pt,a4paper]{amsart}

\makeatletter
\def\specialsection{\@startsection{section}{1}%
  \z@{\linespacing\@plus\linespacing}{.5\linespacing}%
  {\normalfont}}
\def\section{\@startsection{section}{1}%
  \z@{.7\linespacing\@plus\linespacing}{.5\linespacing}%
  {\normalfont\scshape}}
\makeatother

\usepackage{amssymb,amsmath}

\usepackage{microtype}

\usepackage{verbatim}

\usepackage{amssymb,amsmath,esint,enumitem}
\usepackage{hyperref}
\hypersetup{
  colorlinks   = true, 
  urlcolor     = blue, 
  linkcolor    = blue, 
  citecolor   = red 
}

\theoremstyle{plain}
\newtheorem{thrm}{Theorem}[section]
\newtheorem{lemma}[thrm]{Lemma}

\newtheorem{cor}[thrm]{Corollary}
\newtheorem{rmrk}[thrm]{Remark}
\newtheorem{dfn}[thrm]{Definition}

\numberwithin{equation}{section}

\begin{document}
\newcommand{\SL}{\mathcal L^{1,p}( D)}
\newcommand{\Lp}{L^p( Dega)}
\newcommand{\CO}{C^\infty_0( \Omega)}
\newcommand{\Rn}{\mathbb R^n}
\newcommand{\Rm}{\mathbb R^m}
\newcommand{\R}{\mathbb R}
\newcommand{\Om}{\Omega}
\newcommand{\Hn}{\mathbb H^n}
\newcommand{\aB}{\alpha B}
\newcommand{\eps}{\epsilon}
\newcommand{\BVX}{BV_X(\Omega)}
\newcommand{\p}{\partial}
\newcommand{\IO}{\int_\Omega}
\newcommand{\bG}{\boldsymbol{G}}
\newcommand{\bg}{\mathfrak g}
\newcommand{\Bux}{\mbox{Box}}
\title[Variational solution to the Neumann problem]{Existence and uniqueness of variational solution to the Neumann problem for the $p^{th}$ Sub-Laplacian associated to a system of H\"ormander vector fields}

\author{Duy-Minh Nhieu}
\address{Department of Mathematics \\
National Central University \\
Zhongli District, Taoyuan City 32001, \\
Taiwan, R.O.C.}
 \email[Duy-Minh Nhieu]{dmnhieu@math.ncu.edu.tw}
\thanks{Supported by the Ministry of Science and Technology, R.O.C., Grant MOST 106-2115-M-008-013}

\begin{abstract}
We establish the existence and uniqueness of variational solution to the nonlinear Neumann boundary problem for the $p^{th}$-Sub-Laplacian associated to a system of H\"ormander vector fields
\end{abstract}

\maketitle

\section{\textbf{Introduction}}

Two of the fundamental boundary value problems in PDE are the Dirchlet and Neumann problems.  They have profound influences in the development of the theory of PDE.  Aside from their theoretical applications, they describe many physical models.  In the classical setting, these two problems have been advanced to a far extend.  We would not attempt to give a bibliography here since we would inadvertently omit some of the important ones.   Boundary value problems for Sub-elliptic operators were initiated by the pioneering works of Kohn and Nirenberg, Bony, Gaveau and complemented later by two important works of Jerison.  These works established some fundamental aspects of subelliptic equations such as the Harnack inequality , smoothness of the Green's function up to the boundary, for boundaries that are nowhere characteristic with respect to the operator \cite{KN},\cite{B}, explicit formula for the harmonic measure at the center of the Koranyi ball in the Heisenberg group $\Hn$ \cite{G} and the behavior of solutions to the Dirichlet problem near characteristic points \cite{J1,J2}.  Of course, we only highlight those that are closely related to our development in this paper.  For other early relevant results, the reader can see \cite{OL1,OL2,Ba,De}.  These pioneering works subsequently drawed a significance of attention to boundary value problems in the subelliptic setting, see for instance the works \cite{D1},\cite{D2}, \cite{HH}, \cite{Ci}, \cite{Da}, \cite{UL}, \cite{CG}, \cite{GV}, \cite{CGN1}, \cite{CGN2}, \cite{GNg} 
and the references therein.  However, none of the works cited above addressed the Neumann problem.  

Our purpose  is to establish existence and uniqueness of variational solutions of the Neumann problem associated to the non-linear Sub-elliptic $p^{th}$-Laplacian arising from a system of H\"ormander vector fields.   This is a continuation and extension of the works \cite{N} and \cite{DGN}.

To set the stage we fix an open set $U\subset \mathbb R^n$ with $diam(U) < \infty$  and $X = (X_1,...,X_m)$ a system of vector fields with smooth coefficients satisfying H\"ormander's finite rank condition at every point $x\in U$:

\[
rankLie[X_1,...,X_m] \ \equiv \ n
\]
Fix $1 \leq p < \infty$.  Let $\mathcal L_p$ be the $p^{th}$-Sub-Laplacian associated to the system of vector fields $X = (X_1,...,X_m)$ defined on smooth functions $u$ by

\[
\mathcal L_p\,u \ =\ -\, \sum_{j=1}^m X_j^*(|Xu|^{p-2}X_ju)
\]
on $\Rn$. In the above  the $X_j = \sum_{k=1}^n b^{(j)}_k(x) \frac{\partial}{\partial x_k}$ are smooth vector fields on $\Rn$ (satisfying H\"ormander's finite rank condition though such an assumption is not necessary to introduce the operator) and $X^*_j$ is the formal adjoint of $X_j$ given by
$X^*_j g(x) = - \sum_{k=1}^n \frac{\partial}{\partial x_k}(b^{(j)}_k(x) g(x))$.  We have also indicate by $|Xu| = \left(\sum_{j = 1}^m (X_ju)^2\right)^\frac{1}{2}$.  For a smooth (at least $C^2$) domain, the classical Neumann problem for $\mathcal L_p$ consists of seeking a function $u$ that can be differentiated enough times and satisfying

\begin{equation}\label{E:NP}
\begin{cases}
\mathcal L_p u \ =\ f \quad\text{in} \ \Om,\\
\sum_{j=1}^m  |Xu|^{p-2} (X_j \cdot \eta) \,X_ju\ =\ \nu\quad\text{on }\quad\partial{\Om}
\end{cases}
\end{equation}
where $\eta$ is the unit outer (Euclidean) normal of $\partial\Om$, $\nu$ and $f$ are continuous functions satisfy the compatibility condition

\begin{equation}\label{E:Compatibility}
<\nu,1>\ =\ \int_\Om f\,dx\ .
\end{equation}
When $\nu=g$ for some function $g$, then (\ref{E:Compatibility}) simply becomes
\[
\int_{\partial \Om} g\ d\mu \ = \ \int_\Om f\,dx\ .
\]

In the above, the notation $<,>$ indicated the pairing between an element $x\in Z$ and $\nu \in Z^*$ for some Banach space $Z$ and its dual $Z^*$ (we use make use of a more general description of the left hand side of \eqref{E:Compatibility}) in preparation for a more abstract statement later).   For the sake of convenience, we would also indicate the sum $\sum_{j=1}^m (X_ju)(X_j v)$ by the notation $<Xu,Xv>$, as long as it is clear from the context.

In the case $p=2$, the operator $\mathcal L_p$ is known as the (real part of the) Kohn-Laplacian.  Our results here extend the ones established for $p=2$ in \cite{N},\cite{DGN}.  In the classical setting, the method of layer potentials has met success in tackling such problems.  However, even in the simplest prototype of the H\"ormander type vector field, that is, the Heisenberg group, one faces significant difficulties in inverting the corresponding operators on the boundary of the domain (for an appropriate class of functions) due to the presence of characteristic points, see \cite{J1}.  Besides, layer potentials will not help in dealing with the case $p\neq 2$.   In view of this  obstacle, we seek  a different method to establish the problem of existence of the non-linear Neumann problem in this broad context.

One approach to the problem \eqref{E:NP} is to broaden the class and the concept of solutions in such a way that existence can be established.  One would then establish that solutions are regular in some context.  It is our purpose to explore the first aspect in this paper and devote the second one to another occasion \cite{DGMN}.  We now introduce variational solutions.

\vspace{0.3cm}

\begin{dfn}\label{D:NP}
Fix $1 \leq p < \infty, 1 \leq q < \infty$, $f\in L^{q'}(\Om,dx)$, $\nu\in B^q_{1 - \frac{s}{q}}(\partial \Om,d\mu)^*$ where $1/q + 1/q' = 1$.  A function $u\in \mathcal L^{1,q}(\Om,dx)$ is called a variational solution of \eqref{E:NP} if for any $\phi\in\mathcal L^{1,q}(\Om,dx)$

\begin{equation}\label{E:vNP}
\int_{\Om} |Xu|^{p-2}<Xu,X\phi>\,dx\ = \ <\nu,Tr(\phi)>\ -\ \int_\Om f\,\phi\,dx
\end{equation}
In the above, $\Om \subset \Rn$ is a domain for which there is a continuous (trace) operator $Tr:\mathcal L^{1,q}(\Om,dx) \to B^q_{1-\frac{s}{q}}(\partial\Om,d\mu)$.
\end{dfn}
In the above, $\mathcal L^{1,q}(\Omega,dx)$ denotes the by now standard subelliptic Sobolev space associated to the system of vector fields $X$.

For the definition and some properties of the Besov spaces $B^q_{1 - \frac{s}{q}}(\partial \Om,d\mu)$ we refer the reader to Section \ref{S:preliminaries} below.  The measure $\mu$ is an $s$-Ahlfors measure (again see section \ref{S:preliminaries} for the definition) supported in $\partial\Om$ with $0< s < min(q,\frac{n+q}{2})$ 

\vspace{0.3cm}

\begin{rmrk}\label{R:mu}
In order for Definition \ref{D:NP} to make sense, it is required that the continous operator $Tr$ exists and this was established in \cite[Theorem 10.6]{DGN}
under the assumptions

\begin{itemize}
\item[1.] $\Om$ is an extension domain (i.e. $X-(\eps,\delta)$-domain) and
\item[2.] $\mu$ satisfies the estimate (see definition \ref{mu_upper})

\begin{equation}\label{E:Ahf}
 \mu(B(x,r))\ \leq \ M\frac{|B(x,r)|}{r^s}.
\end{equation}

In particular, 1. and 2. are fulfilled when $\Om = \{\rho < 1\}$ where $\rho$
is the homogenous gauge in a Carnot group of step two and when

\[
d\mu = |X\rho|\,d\sigma
\]
where $d\sigma$ is the surface measure on $\partial \Om$, and  $s = 1$ in \eqref{E:Ahf},
see e.g. \cite{CG}, \cite{CGN1} and \cite{CGN2}.  Therefore, in the sequel we will work under
such hypothesis.  Whenever an expression involves function values on the boundary, we mean the trace of such function.  When the context is clear will omit the $Tr$ notation.

\item[3.] By taking the test function $\phi(x)\equiv 1$, we arrive at the compatibility condition \eqref{E:Compatibility}, which is necessary for the existence of solutions to the above variational Neumann problem.

\end{itemize}
\end{rmrk}

\vspace{0.3cm}

The close subspace of $\mathcal L^{1,q}(\Om,dx)$ defined by

\[
\tilde{\mathcal L}^{1,q}(\Om,dx)
\ =\ \left\{f\in\mathcal L^{1,q }(\Om,dx)\,\,\Big |\,\, f_\Om =
\frac{1}{|\Om|}\int_\Om f(x)\,dx\ = 0\right\}
\]
is an appropriate space to treat the Neuamann problem.

\vspace{0.3cm}

\begin{rmrk}\label{R:important}
Due to Poincar\'e inequality (see Theorem \ref{T:Poincare-Sobolev} below), an equivalent norm on $\tilde{\mathcal L}^{1,q}(\Om)$ is given by

\[
\|f\|_{\tilde{\mathcal L}^{1,q}(\Om)}\ = \
\left(\int_\Om |Xf|^q\,dx\right)^\frac{1}{q}.
\]
\end{rmrk}

Our aim is to establish the existence and uniqueness of variational solution in such spaces.  Our main result is the following

\vspace{0.3cm}

\begin{thrm}\label{T:main}
Let $U$ be as in Theorem \ref{T:Sh_restriction_boundary}.
Fix $1<p <\infty$ and $1 < q \leq p$.   Let $f\in L^{q'}(\Om,dx)$, $\Om\subset U \subset \mathbb R^n$ be an $X-(\epsilon,\delta)$ domain with $diam(\Om) < R_o/2$, $|\Om| > 0$, $\mu$ an upper $s$-Ahlfors measure and $\nu\in B^q_{1-s/q}(\partial\Om,d\mu)^*$ satisfying the compatibility condition \eqref{E:Compatibility}.  There exists a unique $u\in\tilde{\mathcal L}^{1,p}(\Om,dx)$ that solves the Neumann problem for $\mathcal L_p$ in the sense of Definition \ref{D:NP}.  In the above, the parameter $R_o$ is from definition \ref{D:Ahlfors} below.  Furthermore, the solution $u$ satisfies the estimate

\begin{equation}\label{E:basic_estimate}
\|Xu\|_{L^p(\Om,dx)}\ \leq\ C\,(\|\nu\|_{B^q_{1-s/q}(\partial \Om,d\mu)^*}\ +\ \|f\|_{L^{q'}(\Om,dx)})^\frac{1}{p-1}\ ,
\end{equation}
where $C$ depends on $\Om$ and various parameters such as $R_o$, $p,q$ and so on but does not depend on $u,\nu,f$.
\end{thrm}

To establish this main result, we show that the functional $J_p$ defined in \eqref{E:FunctionalJ} has a unique minimum and it is the variational solution of \eqref{E:NP}.  Conversely,  variational solutions of \eqref{E:NP} are minimizers of $J_p$. The proof of this fact follows the classical approach of the direct method of Calculus of Variations.  First, we show that the functional $J_p$ is sequentially lower semicontinuous in the weak topology of $\mathcal L^{1,q}(\Om,dx)$.  This is valid for all $1 \leq q < \infty$.  We then establish the coercivity of $J_p$ in $\mathcal L^{1,q}(\Om,dx)$.  At the moment, this is limited to the range $1< q \leq p$.  A weakness as it may seems in comparison to its Euclidean counterpart.  However, in the Euclidean setting, for the case $p=2$ on Lipschitz domains, similar result only holds for $1 < q < 2+\epsilon$ where $\epsilon$ depends on the domain.  The class of domains (the $X-(\epsilon,\delta)$ domains) that we treat here have characteristic points on the boundary.  Such singular points act as corners (or perhaps as cusps) of the domain if seen from the Euclidean perspective.  It is well-known that in the Euclidean case the inclusion of the classes of domains

\begin{equation}\label{E:incl}
\text{Lipschitz} \subset \text{NTA}
\subset (\epsilon,\delta)
\end{equation}
hold.  The so call $(\epsilon,\delta)$ domains were introduce by \cite{Jo} and they are the largest class on which Sobolev functions can be extended.  
It is important to observe th at for a Carnot-Carath\'eodory space the last inclusion of \eqref{E:incl} continue to hold, see \cite{CT}.  A version of Jerison and Kenig's theory on the boundary behaviour of harmonic functions \cite{JK} has been extended to the case of Carnot groups of step two \cite{CG} on $X-NTA$ domains.  Contrary to the Euclidean case where Lipschitz domains are probably the largest class of domains where a rich theory of boundary value problems can be developed, the analogue notion of Lipschitz domain in the subelliptic setting even in the simplest case of the Heisenberg group is almost non-existence, see \cite{CG}.   Therefore part of our task is to identify a class of domains, as large as possible for which a theory of boundary value problems can be developed. The examples in \cite{CT} and \cite{CG} are an indication that the class of $X-(\epsilon,\delta)$ domains treated here is reasonably large.  Due to the inclusion \eqref{E:incl}, the domains considered in this paper, when confined to the Euclidean case,  is larger than the Lipschitz domains.  Hence, it is not supprising if a limitation on the parameter $q$ is occurs.   However, we have not been able to determine the sharp range at this moment. 

We start with some background materials in section \ref{S:preliminaries} and section \ref{S:main} is devoted in establishing our main result.  Finally, we like to point out that, to the best of our knowledge, the results here are new even in the classical setting, as far as the class of domains treated is concerned.

\vspace{0.5cm}

\section{\textbf{Preliminaries}}\label{S:preliminaries}
We collect previously established results needed for subsequent developments.  Some of these results are more general than what we require here.  We present the version that is already adapted to our setting and omit the details that are relevant in a more broader context.

 Let $d$ be the Carnot-Carath\'eodory distance associated with the system $X$.  It is by now well known that if the system $X$ satisfies H\"ormander's finite rank condition \cite{H}, then $d(x,y) < \infty$ for any $x,y\in \mathbb R^n$ (\cite{NSW}, \cite{Chw}, \cite{Rs}) and the metric balls $B(x,r)$ of $d$ satisfy a doubling condition \cite{NSW}.  Denote the Borel measures on the metric space $(\Rn,d)$ by $\mathcal B_d$.  For a set $E\subset \Rn$, $|E|$ denotes the Lebesgue measure of $E$.

\vspace{0.3cm}

\begin{dfn}[\cite{DGN}]\label{D:Ahlfors}
Given $s\geq 0$, a measure $\mu\in \mathcal B_d$ will be called an \emph{upper $s$-Ahlfors measure}, if there exist $M,  R_o > 0$, such that for $x\in \Rn,\ 0<r\leq R_o$, one has

\begin{equation*}\label{mu_upper}
\mu(B(x,r)) \leq\ M\ \frac{|B(x,r)|}{r^s} .
\end{equation*}
\noindent 
We will say that $\mu$ is a \emph{lower $s$-Ahlfors measure},
if for some $M, R_o > 0$ one has instead for $x$ and $r$ as above

\begin{equation*}
\mu(B(x,r))\ \geq\ M^{-1}\ \frac{|B(x,r)|}{r^s} .
\end{equation*}
\end{dfn}
\medskip

The (dual of the) following subelliptic Besov space will serve as the space for the Neumann datum on the boundary.

\vspace{0.3cm}

\begin{dfn}[\cite{DGN}]
Let $\mu\in \mathcal B_d$ having $supp\ \mu \subseteq F$, where $F$ is a closed subset of $\Rn$. For $1 \leq p < \infty$, $0<\beta<1$, we introduce the semi-norm

\[
\mathcal N^p_\beta(f,F, d\mu) \ = \
\left\{\int_F \int_F\ \left(\frac{|f(x)\ -\ f(y)|}{d(x,y)^\beta}\right)^p\ \frac{d(x,y)^s}{|B(x,d(x,y))|}
\,d\mu(y)\,d\mu(x)\right\}^\frac{1}{p} .
\]
The \emph{Besov space} on $F$, relative to the measure $\mu$, is defined as

\[
B^p_\beta(F,d\mu) \ =\ \{f\in L^p(F,d\mu)\,|\, \mathcal N^p_\beta (f,F,d\mu) < \infty\}\ .
\]
If $f\in B^p_\beta(F,d\mu)$, we define the Besov norm of $f$ as

\[
\|f\|_{B^p_\beta(F,d\mu)}\ =\ \|f\|_{L^p(F,d\mu)}\ +\ \mathcal N^p_\beta(f,F,d\mu)\ .
\]
\end{dfn}

We note in passing the following result \cite[Theorem 11.1]{DGN}, which motivates the use of Besov spaces (since their dual are larger) instead of  Lebesgue spaces on the boundary for the Neumann data.  However, our main result here does not make use of Theorem \ref{T:Besov_into_Lq} below.  Note also that since we only needed $\mu$ to be an upper $s$-Ahlfors measure, a comparison between Besov and Lebegues spaces as candidates for the boundary datum also require that we assume $\mu$ to be a lower $s$-Ahlfors measure as well.  In the subelliptic settings, ample supply of measures that satisfy both an upper and lower $s$-Ahlfors condition are found in \cite{DGN} and the references therein.

\vspace{0.3cm}

\begin{thrm}[\textbf{Embedding a Besov space into a Lebesgue space}]
\label{T:Besov_into_Lq}
Given a bounded set $U\subset \Rn$ having characteristic local parameters $C_1,R_o$, and local homogeneous dimension $Q$, let $\Om\subset \overline \Om \subset U$ be an open set with $diam\ \Om < R_o/2$. Let $p\geq 1$, $0 < \beta < 1$. Suppose $\mu$ is a lower $s$-Ahlfors measure with

\begin{equation*}
0< s\ \leq \ n + \beta p ,\quad\quad\quad  s\ <\ Q - \beta p\ ,
\end{equation*}
and such that $supp\ \mu = F \subset \Om$.
There exists a continuous embedding

\[
B^p_\beta(F,d\mu)\ \subset\ L^q(\Om,d\mu) ,\quad\quad\text{where} \quad q\ =\ p\ \frac{Q-s}{Q - s - \beta p} ,
\]
and, in fact, for $f\in B^p_\beta(F,d\mu)$ one has

\begin{equation*}
 \|f\|_{L^q(\Om,d\mu)}
 \ \leq\ C\ \left\{\left(1\ +\ \frac{diam(\Om)^\beta}{\mu(F)^{\beta/(Q-s)}}\right)
\mathcal N^p_\beta(f,F, d\mu)\ +\ \frac{1}{\mu(F)^{\beta/(Q-s)}}\ \|f\|_{L^p(\Om,d\mu)}\right\},
\notag
\end{equation*}
where $C = C(\Om,C_1,R_o,p,\beta,s,M) > 0$.
Furthermore,

\begin{align*}
& \left(\int_\Om |f(x) - f_{\Om,\mu}|^q\,d\mu(x)\right)^\frac{1}{q} \\
\notag
&\qquad\qquad\qquad \leq\ C\ \left(\int_F \int_F |f(x)\ -\ f(y)|^p \frac{d(x,y)^{s-\beta\; p}}{|B(x,d(x,y))|}
\,d\mu(y)\,d\mu(x)\right)^\frac{1}{p} ,
\end{align*}
where $f_{\Om,\mu}$ denotes the average $\frac{1}{\mu(\Om)}\int_{\Om} f\,d\mu$.
\end{thrm}

\medskip

A wide class of domains to which our results hold is the following:

\vspace{0.3cm}

\begin{dfn}[\cite{N,DGN}]
An open set $\Omega \subset \Rn$ is called an $X-(\eps,\delta)$-domain if
there exist $0<\delta\leq \infty,\,0<\eps ~\leq ~1$, such that for any pair of
points $p,q \in \Omega $, if $d(p,q) \leq \delta$, then one can find a
continuous, rectifiable curve $\gamma :[0,T]\to\Omega$, for which
$\gamma(0)=p,\,\gamma(T)=q$, and

\begin{equation*}
l(\gamma )\ \le\ \frac{1}{\eps}\ d(p,q),\quad\quad d(z, \partial\Omega)\ \geq\ \eps\ \min\ \{d(p,z),d(z,q)\} \quad\text{for all } z\in \{\gamma\}.
\end{equation*}
\end{dfn}
\medskip

The following compact embedding theorem \cite[Theorem 1.28]{GN} plays an important role in our proof.

\vspace{0.3cm}

\begin{thrm}[\textbf{Compact embedding}]\label{T:Rellich}
Let $\Om\subset\Rn$ be an $X-(\epsilon,\delta)$
domain with $diam(\Om) <\frac{R_0}{2}$.  Then, one has the following:

\begin{itemize}
\item[(I)] The embedding $BV_X(\Om,dx)\hookrightarrow L^q(\Om,dx)$ is compact
for any $1\leq q < \frac{Q}{Q-1}$.
\smallskip
\item[(II)] For any $1\leq p < Q$ the
embedding $\mathcal L^{1,p}(\Om,dx)\hookrightarrow L^q(\Om,dx)$ is compact provided
that $1\leq q < \frac{Qp}{Q-p}$.
\smallskip
\item[(III)] For any $Q\leq p < \infty$
and any $1\leq q < \infty$, the embedding $\mathcal L^{1,p}(\Om,dx)\hookrightarrow
L^q(\Om,dx)$ is compact.
\end{itemize}
\end{thrm}

As an easy consequence of the above Theorem we have

\vspace{0.3cm}

\begin{thrm}\label{T:cmpct-emb}
With the same assumptions in Theorem \ref{T:Rellich}, for any $1 \leq p < \infty$ the embedding  $\mathcal L^{1,p}(\Om,dx)\hookrightarrow
L^p(\Om,dx)$ is compact.
\end{thrm}

We can deduce from Theorem \ref{T:cmpct-emb} easily the following

\vspace{0.3cm}

\begin{cor}\label{C:strong-conv}
For any $1 < p <\infty$, if $\{u_h\}_{h=1}^\infty$ is a sequence in $\mathcal L^{1,p}(\Om,dx)$ such that $u_h \rightharpoonup u$ for some $u\in \mathcal L^{1,p}(\Om,dx)$ then $u_h \to u$ in $L^p(\Om,dx)$ hence $u_h \to u$ in $L^1(\Om,dx)$ also.  
(We note explicitly the notation $\rightharpoonup$ means weak convergence whereas $\to$ means convergence in norm).
\end{cor}

The proof of this fact is rather standard but we include it here for the sake of convenience of the reader.

\begin{proof}[\textbf{Proof}]
The weak convergence assumption together with the Banach-Steinhaus (uniform boundedness principle) theorem implies that $\{u_h\}_{h=1}^\infty$ is bounded in $\mathcal L^{1,p}(\Om,dx)$.  Now Theorem \ref{T:cmpct-emb} implies that if $\{u_{h_j}\}$ is any subsequence of $\{u_h\}$, $\{u_{h_j}\}$ has a subsequence we denote it by $\{v_{h_j}\}$ such that $v_{h_j} \to v \in L^p(\Om,dx)$ for some function $v$.  Since $u_h \rightharpoonup u$ in $L^p(\Om,dx)$, we must have $u = v$ and the original sequence $\{u_h\}$ converge in norm to $u$ in $L^p(\Om,dx)$.
\end{proof}

The traces of Sobolev functions on the boundary of a domain is a delicate matter.  It was the purpose of \cite{DGN} to develop such a theory in the setting of a Carnot-Carath\'eodory space.  It is also an indispensible tool in dealing with the Neumann problem, even the non-linear version.  We thus recall  \cite[Theorem 10.6]{DGN}:

\vspace{0.3cm}

\begin{thrm}[\textbf{Trace theorem on the boundary}]
\label{T:Sh_restriction_boundary}
Let $U\subset \Rn$ be a bounded set with characteristic local parameters $C_1, R_o$, and let $p>1$. There is $\sigma = \sigma(X,U)>0$ such that, if $\Om\subset U$ is a bounded
$X-(\epsilon,\delta)$-domain with $rad(\Om)>0, diam(\Om) < \frac{R_o}{2\sigma}$,
$dist(\Om,\partial U) >  R_o$, and $\mu$ is an upper $s$-Ahlfors measure for some $0<s<p$,
having $supp\ \mu\subseteq \partial\Om$, then for every $0<\beta \leq 1 - s/p$
there exist a linear operator

\[
\mathcal Tr\ :\ \mathcal L^{1,p}(\Om,dx)\ \to\  B^p_\beta(\partial\Om,d\mu)\ ,
\]
and
a constant $C=C(U,p,s, M, \beta,\epsilon,\delta,rad(\Om))>0$, such that

\begin{equation}\label{E:Besov_ineq_boundary}
\|Tr\,f\|_{B^p_\beta(\partial\Om,d\mu)}
\ \leq\ C\ \|f\|_{\mathcal L^{1,p}(\Om,dx)}\ .
\end{equation}
Furthermore, if $f\in C^\infty(\overline\Om)\cap\mathcal L^{1,p}(\Om,dx)$, then $\mathcal Tr\,f = f$ on $\partial\Om$.
\end{thrm}

Our final pillar is the following result \cite[Corollary 1.6]{GN} specialized to our setting:

\vspace{0.3cm}

\begin{thrm}\label{T:Poincare-Sobolev}
Let $\Om$ be an $X-(\epsilon, \delta)$ domain, $1\leq p < Q$.  There exists some constant $C=C(\epsilon,\delta,p,n)>0$ such that
 for any $1\leq k \leq \frac{Q}{Q-p}$ and any $u\in\mathcal L^{1,p}(\Om)$

\begin{equation*}
\left(\frac{1}{|\Om|}\int_\Om |u-u_\Om|^{kp}\,dx\right)^\frac{1}{kp}
\leq C \ diam(\Om)\left(\frac{1}{|\Om|}\int_\Om |Xu|^p\,dx\right)^\frac{1}{p}
\end{equation*}
\end{thrm}

\vspace{0.5cm}

\section{\textbf{Existence and Uniqueness of solutions to the Neumann Problem}}\label{S:main}

Let $\Om$ be an $X-(\epsilon,\delta)$ domain.  Consider the functional $J_p:\mathcal L^{1,q}(\Om,dx) \to \overline R$ given by

\begin{equation}\label{E:FunctionalJ}
J_p(u)\ =\ \int_{\Om} \frac{1}{p} |Xu(x)|^p  +  f(x)\,u(x)\,dx\ - \ <\nu,Tr(u)>\ .
\end{equation}

Note that in appriori, unless $q = p$ or $u\in C^\infty(\overline{\Omega})\cap \mathcal L^{1.q}(\Omega,dx)$, the functional $J_p$ can take on the value of $\pm\infty$ and the assumption that $\Om$ is an $X-(\epsilon,\delta)$ domain is necessary for the trace operator $tr$ (as in Definition \ref{D:NP}) to be defined.  Hence, this assumption is needed in the above definition and throughout the paper.

\vspace{0.3cm}

\begin{lemma}\label{L:convex}
Let $\Om$ be an $X-(\epsilon,\delta)$ domain.  For any $q \geq 1$ the functional $J_p$ given by \eqref{E:FunctionalJ} is convex on $\mathcal L^{1,q}(\Om,dx) $.  If in addition we assume that $|\Om| > 0$ and $p > 1$ then the functional $J_p$ is strictly convex in the subspace $\tilde{\mathcal L}^{1,q}(\Om,dx)$.
\end{lemma}

\begin{proof}[\textbf{Proof}]
It suffices to establish the proposition for the non-linear part of $J_p$, namely

\begin{equation}\label{E:non-linear}
I(u)\ =\  \frac{1}{p}\, \int_{\Om}  |Xu(x)|^p\,dx\ .
\end{equation}

It is an elementary fact that the function $g:\mathbb R^m \to \mathbb R$ given by $g(z) = |z|^p$ is convex for $1\leq p < \infty$ (see e.g., \cite[Theorem 5.1]{Roc} and the remark there).  This implies that any functional of the form
$\mathcal F(u) = \int_\Om g(Xu)(x)\,dx$ is convex.  Hence, $I$ is convex.  We turn to the second part of the proposition.  For $p > 1$ the function $g$ is strictly convex.  Hence for any $0 < t < 1$ and any $u,v\in \tilde{\mathcal L}^{1,q}(\Om,dx)$ we have $g(t(Xu) + (1-t)(Xv)) < tg(Xu)|(1-t)g(Xv)$ unless $Xu = Xv$, that is, for all $i=1,..,m,\,X_i(u-v)=0$.   To continue, observe that the system of H\"ormander vector field $X = (X_1,...,X_m)$ satisfy the property that

\begin{equation*}
\forall i = 1,..m, \quad\text{for} \quad a.e.\ x\in\Om:\, X_i g(x) = 0 \quad\text{implies} \quad g(x) = constant \quad\text{for} \quad a.e.\ x \in \Om \ .
\end{equation*}
This imply $u = v + c$ for some constant $c$.  Since $u,v\in \tilde{\mathcal L}^{1,p}(\Om,dx)$, taking into account that $|\Om| > 0$, we see that $c$ must be zero.
\end{proof}

\vspace{0.3cm}

\begin{rmrk}\label{R:Compatibility}
Two remarks are in order.
\newline
(i)\, The uniqueness of minimizer would follow from the strict convexity of $J_p$, that is, minimizer would be unique in the spaces $\tilde{\mathcal L}^{1,q}(\Om,dx)$.
\newline
(ii)\, Due to the compatibility condition \eqref{E:Compatibility}, we have for any constant $c\in\R$ and any $u\in \mathcal L^{1,q}(\Om,dx)$, $J_p(u + c) = J_p(u)$.  Hence, if a (unique) minimizer  $u_o\in\tilde{\mathcal L}^{1,q}(\Om,dx)$ exists for $J_p$, then $u_o$ remains to be a minimizer for $J_p$ in $\mathcal L^{1,q}(\Om,dx)$ since for any $u\in \mathcal L^{1,q}(\Om,dx)$, $J_p(u) = J_p(u - u_\Om) \geq J_p(u_o)$ and that for any constant $c\in\R$, $u_o + c$ is also a minimizer for $J_p$ in $\mathcal L^{1,q}(\Om,dx)$.
\end{rmrk}

Next, we turn to the (sequential) lower semicontinuity of the functional $J_p$.   The following result (in many equivalent form) from Calculus of Variations is valuable to us and is available  from many sources.  For the sake of our purpose, we apply the following theorem from \cite[Theorem 4.5]{Gi}.

\vspace{0.3cm}

\begin{thrm}\label{T:lsc}
Let $\Om$ be an open set in $\R^n$, $M$ a closed set in $\R^N$, and let $F(x,u,z)$ be a function defined in $\Om\times M \times \R^l$ such that

\begin{enumerate}
\item[(i)] $F$ is a Caratheodory function, that is measurable in $x$ for every $(u,z)\in M\times\R^l$ and continuous in $(u,z)$ for almost every $x\in\Om$
\item[(ii)] $F(x,u,z)$ is convex in $z$ for almost every $x\in\Om$ and for every $u\in M$.
\item[(iii)] $F \geq 0$.
\end{enumerate}
Let $u_h,u\in L^1(\Om,M,dx),z_h, z \in L^1(\Om,\R^l)$ and assume that $u_h \to u$ and $z_h \rightharpoonup z$ in $L^1_{loc}(\Om,dx)$.  Then,

\[
\int_\Om F(x,u,z)\,dx \ \leq \liminf_{h\to \infty} \,\int_\Om F(x,u_h,z_h)\,dx
\]
\end{thrm}

\vspace{0.3cm}

\begin{cor}\label{C:lsc}
Let $\Om$ be an $X-(\eps,\delta)$ domain with $diam(\Om) < \frac{R_o}{2}$ (and $f \in L^{q'}(\Om,dx)$) .
The functional $J_p$ defined in \eqref{E:FunctionalJ} is (sequentially) lower semicontinuous on the spaces $\mathcal L^{1,q}(\Om,dx)$ for any $1\leq q < \infty$ with respect to the weak topology, that is, if $\{u_h\}_{h = 1}^\infty$ is a sequence in $\mathcal L^{1,q}(\Om,dx)$ such that $u_h \rightharpoonup u$ for some $u\in \mathcal L^{1,q}(\Om,dx)$ then

\[
J_p(u)\ \leq\ \liminf_{h\to\infty} J_p(u_h)\ .
\]
\end{cor}

\begin{proof}[\textbf{Proof}]
By assumption $u\in \mathcal L^{1,q}(\Om,dx), f\in L^{q'}(\Om,dx)$ and $\nu\in B^q_{1 - \frac{s}{q}}(\partial \Om,d\mu)^*$  the linear part of $J$ is bounded by H\"older's inequality and the trace inequality \eqref{E:Besov_ineq_boundary} therefore continuous.  The sum of two lower semicontinuous function is lower semicontinuous.  Hence, it suffices to establish the lower semicontinuity of the non-linear part of $J_p$, namely $I$ defined in \eqref{E:non-linear}.  Let $u, \{u_h\}_{h=1}^\infty$ be as in the hypothesis.  By Corollary \ref{C:strong-conv}  we have $u_h \to u$ in $L^1(\Om,dx)$ and also  $X_j u_h \rightharpoonup X_j u$ in $L^q_{loc}(\Om,dx)$ hence also in $L^1(\Om,dx)$ for $j=1,..m$.  Now
we apply Theorem \ref{T:lsc} with $N=1,M=\R,l = m$, $z = Xu = (X_1u,...,X_mu)$ and $F(x,u,z) = |z|^p$ to reach the conclusion.
\end{proof}

\vspace{0.3cm}

\begin{lemma}\label{L:min=sol}
Let  $1 \leq q < \infty$, $f \in L^{q'}(\Om,dx)$, $\Om$ be $X-(\epsilon,\delta)$ domain.  A function $u\in\mathcal L^{1,q}(\Om,dx)$ minimizes  the functional $J_p$ given by \eqref{E:FunctionalJ} if and only if it is a  variational solution to the Neumann problem \eqref{E:NP} in the sense of Definition \ref{D:NP}.
\end{lemma}

\begin{proof}[\textbf{Proof}]
Let $u \in\mathcal L^{1,q}(\Om,dx)$ be a minimizer of $J_p$.  We have for any $\epsilon > 0$ and any $\phi \in\mathcal L^{1,q}(\Om,dx)$ the function $\psi(\epsilon) = J_p(u + \epsilon\,\phi)$ reaches a local minimum at $\epsilon = 0$. Therefore $\psi'(0) = 0$.  On the other hand, differentiating under the integral sign yields

\[
\psi'(\epsilon)\ =\ \int_\Om |Xu +\epsilon\,X\phi|^{p-2}<Xu+\epsilon\,X\phi,X\phi>\ +\ f(x)\phi(x) \,dx\ -\ <\nu,\phi>\ .
\]
Hence

\[
0\ =\ \psi'(0)\ =\ \int_\Om |Xu|^{p-2}<Xu,X\phi>\ +\ f(x)\phi(x)\, dx\ -\ <\nu,\phi>\ ,
\]
which is \eqref{E:vNP} in Definition \ref{D:NP}.

We now prove the converse.  Let $u\in \mathcal L^{1,q}(\Om,dx)$ be a variational solution in the sense of Definition \ref{D:NP}.  For any $v\in \mathcal L^{1,q}(\Om,dx)$ we let $\phi = u - v \in \mathcal L^{1,q}(\Om,dx)$ in \eqref{E:vNP}, adding and subtracting the terms $\int_\Om \frac{1}{p'} |Xu|^p\,dx$, $\int_\Om \frac{1}{p} |Xv|^p\,dx$ we obtain the identity

\begin{align*}
J_p(u) & \ =\ J_p(v)\ -\ \int_\Om \frac{1}{p}|Xv|^p + \frac{1}{p'}|Xu|^p - |Xu|^{p-2}<Xu,Xv>\,dx \\
&\ =\
J_p(v)\ -\ \int_\Om \frac{1}{p} |Xv|^p - \frac{1}{p} |Xu|^p - |Xu|^{p-2}<Xu,Xv - Xu>\,dx \ .
\end{align*}
To show that the last term in the above is non-negative, we recall that the convexity of the function $g(z) = |z|^p$ for $p \geq 1$ can be rephrase as: For every $z,w \in \mathbb R^m$, $z\neq 0$ then

\begin{equation}\label{E:p-convex}
\frac{1}{p}\, |w|^p\ \geq\ \frac{1}{p} |z|^p \ +\ |z|^{p-2}\,\sum_{i=1}^m z_i\,(w_i - z_i)\ .
\end{equation}

Letting $z_i = X_iu$ and $w_i=X_iv$ in the inequality \eqref{E:p-convex} and integrating we have

\[
\int_\Om \frac{1}{p} |Xv|^p\ -\ \frac{1}{p} |Xu|^p\ -\ |Xu|^{p-2}\,<Xu,Xv - Xu>\,dx \ \geq\ 0\ .
\]
Hence, we have $J_p(u) \leq J_p(v)$ and this completes the proof.
\end{proof}

Our final result leading to the existence of minimizers of $J_p$ (hence, variational solutions to the Neumann problem in $\mathcal L^{1,q}(\Om,dx)$) is the following

\vspace{0.3cm}

\begin{lemma}\label{L:coercivity}
Let $\Om$ be an $X-(\epsilon,\delta)$ domain.  For any $1\leq q \leq  p$ and $f \in L^{q'}(\Om,dx)$, the functional $J_p$ is coercive in $\tilde{\mathcal L}^{1,q}(\Om,dx)$, that is, if $\{u_h\} $ is a sequence in $\tilde{\mathcal L}^{1,q}(\Om,dx)$ such that $\lim_{h\to\infty} \|u_h\|_{\tilde{\mathcal L}^{1,q}(\Om,dx)} =
\lim_{h\to\infty} \|Xu_h\|_{L^q(\Om,dx)} = \infty$ then $\lim_{h\to\infty} J_p(u_h) = \infty$.
\end{lemma}

\begin{proof}[\textbf{Proof}]
Observe that due to the Trace theorem \ref{T:Sh_restriction_boundary} and the fact that $f\in L^{q'}(\Om,dx)$, the linear part of $J_p$ is bounded.  Therefore, it suffices to establish the Lemma for the non-linear part of $J_p$, namely $I$ given by \eqref{E:non-linear}.
Let $\{u_h\} $ be a sequence in $\tilde{\mathcal L}^{1,q}(\Om,dx)$ such that
$\lim_{h\to\infty} \|Xu_h\|_{L^q(\Om,dx)} = \infty$.  By H\"older's inequality, we have for $1 \leq q \leq p$ that $\|Xu_h\|_{L^q(\Om,dx)}\leq |\Om|^{\frac{1}{q} - \frac{1}{p}}\|Xu_h\|_{L^p(\Om,dx)}$ hence $\lim_{h\to\infty} I(u_h) = \lim_{h\to\infty} \|Xu_h\|_{L^p(\Om,dx)} = \infty$ as well.
\end{proof}

We now come to our main result which follows from a standard line of argument of the Calculus of Variations, Lemma \ref{L:coercivity} and Corollary \ref{C:lsc}.  We include the proof for the sake of completeness.

\vspace{0.3cm}

\begin{thrm}\label{E:existence_minimizers}
Let $1<q \leq p$, $\Om$ be an $X-(\epsilon,\delta)$ domain with $diam(\Om) < R_o/2$, $|\Om| > 0$, $f\in L^{q'}(\Om,dx)$, $\nu\in B^q_{1 - \frac{s}{q}}(\partial \Om,d\mu)^*$ satisfying the compatibility condition \eqref{E:Compatibility}.  The functional $J_p$ given by \eqref{E:FunctionalJ} has a unique minimizer in $\mathcal L^{1,q}(\Om,dx)$.  
\end{thrm}

\begin{proof}[\textbf{Proof}]
Let $l = inf\{J_p(u)\,|\, u \in \tilde{\mathcal L}^{1,q}(\Om,dx)\} > -\infty$ since the linear part of $J_p$ is bounded and the non-linear part $I$ given by \eqref{E:non-linear} is non-negative.  Let $\{u_h\} \subset \tilde{\mathcal L}^{1,q}(\Om,dx)$ be a minimizing sequence, that is $\lim_{h\to \infty} J_p(u_h) = l$.  Clearly, $J_p(u_h)$ is bounded and hence Lemma \ref{L:coercivity} implies that $\{u_h\}$ is a bounded sequence in $\tilde{\mathcal L}^{1,q}(\Om,dx)$.  Since for $q > 1$, the balls in $\tilde{\mathcal L}^{1,q}(\Om,dx)$ are weakly compact, $\{u_h\}$ contains a subsequence (still denoted by $\{u_h\}$) and $u\in \tilde{\mathcal L}^{1,q}(\Om,dx)$ such that $u_h \rightharpoonup u_o$.  Corollary \ref{C:lsc} then imply

\[
l\ \leq\ J_p(u_o)\ \leq\ \liminf_{h\to\infty} J_p(u_h)\  = \lim_{h\to\infty} u_h\ =\ l\ ,
\]
hence $J_p(u_o) = l$.  Now from Remark \ref{R:Compatibility}, $u_o$ remains to be a minimizer of $J_p$ in $\mathcal L^{1,q}(\Om,dx)$.
\end{proof}

Finally, we now come to the proof of our main result.

\begin{proof}[\textbf{Proof of Theorem \ref{T:main}}]
Theorem \ref{E:existence_minimizers} and Lemma \ref{L:min=sol} yield a unique solution $u\in \tilde{\mathcal L}^{1,q}(\Om,dx)$ in the sense of Definition \ref{D:NP}.  To establish the estimate \eqref{E:basic_estimate} we take $\phi = u$ in Definition \ref{D:NP} and recalling Remark \ref{R:important}  we have

\begin{align*}
\|Xu\|^p_{L^p(\Om,dx)}
&\ =\
\int_\Om |Xu|^{p-2}\,<Xu,Xu>\,dx \\
&\ \leq\
|<\nu,tr(u)>| \ +\ \left|\int_\Om f\,u\,dx\right|\\
&\ \leq\
\|\nu\|_{B^q_{1-s/q}(\partial \Om,d\mu)^*} \, \|tr(u)\|_{B^q_{1-s/q}(\partial\Om,d\mu)} \\
&\qquad \ +\ \|f\|_{L^{q'}(\Om,dx)}\,
\|u\|_{L^q(\Om,dx)} \\
\text{(by Theorems \ref{T:Sh_restriction_boundary},\ref{T:Poincare-Sobolev})}
&\ \leq\
C\,\|\nu\|_{B^q_{1-s/q}(\partial \Om,d\mu)^*} \, \|u\|_{ \tilde{\mathcal L}^{1,q}(\Om,dx)} \\
&\qquad
\ +\ \|f\|_{L^{q'}(\Om,dx)}\,\|Xu\|_{L^q(\Om,dx)} \\
\text{(by Remark \ref{R:important})}
&\ \leq\
C\,(\|\nu\|_{B^q_{1-s/q}(\partial \Om,d\mu)^*} 
\ +\ \|f\|_{L^{q'}(\Om,dx)})\,\|Xu\|_{L^q(\Om,dx)} \\
\text{(since $q \leq p$)}
&\ \leq\
C\,(\|\nu\|_{B^q_{1-s/q}(\partial \Om,d\mu)^*} 
\ +\ \|f\|_{L^{q'}(\Om,dx)})\,\|Xu\|_{L^p(\Om,dx)} \\
\end{align*}
and therefore

\[
\|Xu\|_{L^p(\Om,dx)}\ \leq\ C\,(\|\nu\|_{B^q_{1-s/q}(\partial \Om,d\mu)^*}\ +\ \|f\|_{L^{q'}(\Om,dx)})^\frac{1}{p-1}\ .
\]

In view of Remark \ref{R:important} again this shows $u \in \tilde{\mathcal L}^{1,p}(\Om,dx)$ and the proof is now complete.
\end{proof}

\vspace{0.3cm}

\begin{rmrk}
Some intermediate results such as Lemmas \ref{L:convex}, \ref{L:min=sol} and Corollary \ref{C:lsc} hold for any $q \geq 1$.  The limitation of $1< q \leq p$ comes in Lemma \ref{L:coercivity} and therefore also in Theorem \ref{T:main}.  The case for $q < p$ can also be infered from the case $q=p$ due to the containment of the spaces $B^q_{1-s/q}(\partial\Omega,d\mu)^* \subset B^p_{1-s/p}(\partial\Omega,d\mu)^*$ and $\tilde{\mathcal L}^{1,p}(\Omega,dx)\subset \tilde{\mathcal L}^{1,q}(\Omega,dx)$.
\end{rmrk}

\end{document}